\documentclass[12pt,a4paper]{article}
\usepackage{amssymb}
\usepackage{amsmath}
\usepackage{amsthm}

\newtheorem{theorem}{Theorem}
\newtheorem{definition}{Definition}
\newtheorem{proposition}{Proposition}
\newtheorem{example}{Example}
\newtheorem{remark}{Remark}
\newtheorem{lemma}{Lemma}

\title{A decision functional for \\multitime controllability}

\author{Cristian Ghiu and Constantin Udri\c ste}

\date{}

\begin{document}

\maketitle

\begin{center}
University Politehnica of Bucharest,
Faculty of Applied Sciences, Department of Mathematics II,
Splaiul Independen\c tei 313, 060042 Bucharest,
Romania, e-mail: crisghiu@yahoo.com
\end{center}

\begin{center}
University Politehnica of Bucharest, Faculty of Applied Sciences,
Department of Mathematics-Informatics I,
Splaiul Independen\c tei 313, 060042 Bucharest, Romania,
e-mail: udriste@mathem.pub.ro, anet.udri@yahoo.com
\end{center}

\begin{abstract}
This paper investigates the multitime linear normal PDE systems. We study especially
the controllability of such systems, obtaining complementary results
to those in our recent papers. Here the multitime controllability original
results are formulated using the
$\gamma$ - gramian matrix, the Im - gramian space and a controllability functional.
There are given also some original examples which illustrate and
round the theoretical results.
\end{abstract}



\section{Introduction}

This article studies the controllability of
multitime linear PDE systems using new ingredients.
Our results in this direction are complementary to those contained in the papers
\cite{8}, \cite{9}, \cite{1} -- \cite{3}, \cite{6}, \cite{7}, \cite{20}. The
multitime optimal control, and especially the multitime maximum principle,
was developed in \cite{4}, \cite{5}, \cite{10} -- \cite{18}, \cite{19}.

We shall introduce and study new concepts, as for example:

-- the {\it Im - gramian space}; it is a generalization of the controllability gramian image
which was defined in the papers \cite{8}, \cite{3} only in certain supplementary conditions
(the relations (\ref{II6}) from this paper),

-- the {\it controllability $\gamma$ - functional} and the {\it controllability functional};
they appear mixing our ideas with those of \cite{20}, where a similar
single-time functional is presented.

The basic original results of this paper are the Theorems \ref{tteorema6} (of Section 4),
\ref{tteorema7} (of Section 5) and \ref{tteorema8} (of Section 6). The Theorem \ref{tteorema6} gives necessary conditions of multitime controllability, expressed by the {\it Im - gramian space}.
The Theorem \ref{tteorema7} studies necessary conditions of multitime controllability, expressed
by the {\it controllability $\gamma$ - functional}. The Theorem \ref{tteorema8} contains necessary and sufficient conditions of multitime controllability, expressed by
the {\it controllability functional}.

\section{Preliminary results}

\subsection{Controllability of multitime linear PDE systems}

Let $D \subseteq \mathbb{R}^m$ be an open and convex
subset. We consider the evolution PDE system
\begin{equation}
\label{II3}
\frac{\partial x}{\partial t^{\alpha}}
=
M_{\alpha}(t)x+N_{\alpha}(t)u_{\alpha}(t),
\quad
\forall \alpha=\overline{1,m},
\end{equation}
where $t=(t^1,\ldots,t^m)\in \mathbb{R}^m$, called {\it multitime}, and
$x=(x^1,\ldots,x^n)^{\top}: D \to \mathbb{R}^n=\mathcal{M}_{n,1}(\mathbb{R})$. Also
$M_{\alpha}:D \to \mathcal{M}_{n}(\mathbb{R})$ are $\mathcal{C}^1$
quadratic matrix functions, $N_{\alpha}:D \to \mathcal{M}_{n,k}(\mathbb{R})$
are $\mathcal{C}^1$ rectangular matrix functions and
$u_{\alpha}:D \to\mathbb{R}^k=\mathcal{M}_{k,1}(\mathbb{R})$ are
$\mathcal{C}^1$ {\it vector control functions},
all indexed after $\alpha=\overline{1,m}$.

The PDE system (\ref{II3}) is called {\it completely integrable} if
$\forall (t_0,x_0)\in D \times \mathbb{R}^n$, there
exists an open set $D_0 \subseteq D \subseteq \mathbb{R}^m$, with
$t_0\in D_0$ and $\exists x:D_0\to \mathbb{R}^n$,
$x(\cdot)$ differentiable, such that $x(\cdot)$ verifies
equations (\ref{II3}) on $D_0$ and $x(t_0)=x_0$. In this
case $x(\cdot)$ will be called a {\it solution}
for Cauchy problem $\{ (\ref{II3}), \, x(t_0)=x_0 \}$.

The system $(\ref{II3})$ is
completely integrable if and only if the
following relations 
\begin{equation}
\label{II4}
\begin{split}
\frac{\partial M_{\alpha}}{\partial t^{\beta}}
+
M_{\alpha}(t)M_{\beta}(t)
=
\frac{\partial M_{\beta}}{\partial t^{\alpha}}
+
M_{\beta}(t)M_{\alpha}(t),
\end{split}
\end{equation}
\begin{equation}
\label{II5}
\begin{split}
& M_{\alpha}(t)N_{\beta}(t)u_{\beta}(t)
+
\frac{\partial N_{\alpha}}{\partial t^{\beta}}
u_{\alpha}(t)+
N_{\alpha}(t)
\frac{\partial u_{\alpha}}{\partial t^{\beta}} \\
&=
M_{\beta}(t)N_{\alpha}(t)u_{\alpha}(t)
+
\frac{\partial N_{\beta}}{\partial t^{\alpha}}
u_{\beta}(t)+
N_{\beta}(t)
\frac{\partial u_{\beta}}{\partial t^{\alpha}},
\end{split}
\end{equation}
hold, $\forall t\in D, \,\, \forall \alpha,\beta=\overline{1,m}$.
In these conditions, any solution $x(\cdot)$ will be
a $\mathcal{C}^2$ function and can be uniquely
extended to a global solution
$\left ( \displaystyle x: D \to \mathbb{R}^n \right )$. If
two solutions coincide at a point, then they will
coincide on whole set $D$.

Thereafter, by a solution we mean a global solution. In the papers
\cite{8}, \cite{3}, it was shown that if the relations (\ref{II4}), (\ref{II5}) are true,
then the solution of the Cauchy problem $\{ (\ref{II3}), \, x(t_0)=x_0 \}$ is

\begin{equation}
\label{solCauchy}
x(t)=
\chi(t,t_0)x_0
+
\int\limits_{\gamma_{t_0,t}}^{}
\chi(t,s)N_{\alpha}(s)
u_{\alpha}(s)
\mbox{d}s^{\alpha},
\quad
\forall t \in D,
\end{equation}
where $\gamma_{t_0,t}$ is a piecewise $\mathcal{C}^1$ curve included in $D$,
traversed from the multitime $t_0$ to the multitime $t$ and $\chi(t,s)$
is the fundamental matrix associated to the PDE system, i.e., the matrix solution of
the Cauchy problem (see \cite{8})
\begin{equation*}
\frac{\partial \chi}{\partial t^{\alpha}}
(t,s)
=
M_{\alpha}(t)\chi(t,s), \quad\chi(s,s)=I_n,\quad \forall \alpha=\overline{1,m}.
\end{equation*}

\begin{definition}
\label{ddefinitia1}
\rm
Suppose that the matrix functions
$M_\alpha(\cdot)$ verify the relations $(\ref{II4})$,
$\forall t\in D, \,\, \forall \alpha,\beta=\overline{1,m}$.
The vector space
$$
\mathcal{U}= \Bigm \{u=(u_{\alpha})_{\alpha=\overline{1,m}}
\Bigm |
u_{\alpha}:D \to \mathbb{R}^k=\mathcal{M}_{k,1}(\mathbb{R}),
\mbox{of class}\,\, \mathcal{C}^1, \forall \alpha=\overline{1,m}
$$
$$
\mbox{ and which verify the relations $(\ref{II5})$ for all } \alpha, \beta
\Bigm \}
$$
is called the {\it control space},
associated to the system $(\ref{II3})$.
If $u\in \mathcal{U}$, we say that $u$ is a {\it control}.
\end{definition}

So, if the matrices
$M_{\alpha}(\cdot)$ verify the relations $(\ref{II4})$,
$\forall t \in D$ and $\forall \alpha, \beta=\overline{1,m}$,
then the system (\ref{II3}) is completely integrable if and
only if $(u_{\alpha})_{\alpha=\overline{1,m}}$ is
a control function.

\begin{definition}
\label{ddefinitia2}
\rm
Let us consider the PDE system $(\ref{II3})$, with the matrix functions
$M_\alpha(\cdot)$ verifying the relations $(\ref{II4})$.


$a)$ The phase $(t,x)$ is called {\it controllable}
if there exists a point $s\in D$, with $s^{\alpha}>t^{\alpha}$, $\forall \alpha$,
and there exists a control $u(\cdot)$ which transfers the phase $(t,x)$ into the phase $(s,0)$.

$b)$ Let $t_0,t \in D$, with $t_0^{\alpha}<t^{\alpha}$, $\forall \alpha$.
The PDE system $(\ref{II3})$ is called {\it completely controllable} from $t_0$ to $t$
if for any point $x \in \mathbb{R}^n$, the phase $(t_0,x)$ transfers into the phase $(t,0)$, i.e., for any point $x$ the phase $(t_0,x)$ is controllable with the same $t$.

\end{definition}

Let us consider the PDE system $(\ref{II3})$, with the matrix functions
$M_\alpha(\cdot)$ verifying the relations $(\ref{II4})$. Taking $t_0,t\in D$,
we consider the set
$$
\mathcal{V}(t_0,t):=
\Bigg \{
\int\limits_{\gamma_{t_0,t}}^{}
\chi(t_0,s)N_{\alpha}(s)
u_{\alpha}(s)
\mbox{d}s^{\alpha}
\,
\Big|
\,
(u_{\alpha})_{\alpha=\overline{1,m}}
\mbox{\,\, is a control }
\Bigg \}.
$$

If $(u_{\alpha})_{\alpha=\overline{1,m}}$\, is
a control, then the curvilinear integral
$$\displaystyle
\int\limits_{\gamma_{t_0,t}}^{}
\chi(t_0,s)N_{\alpha}(s)
u_{\alpha}(s)
\mbox{d}s^{\alpha}
$$
is path independent, so $\mathcal{V}(t_0,t)$
does not depend on the curve $\gamma_{t_0,t}$,
which joins $t_0$ to $t$, but it depends on the multitimes $t_0$ and $t$.

One remarks immediately that the set $\mathcal{V}(t_0,t)$ is a vector subspace
of $\mathbb{R}^n$.

\begin{definition}
\label{ddefinitia3}
\rm
The space $\mathcal{V}(t_0,t)$ is called the {\it controllability space}.
\end{definition}

\begin{theorem}
\label{tteorema5}
{\it Let us consider the PDE system $(\ref{II3})$, with the matrix functions
$M_\alpha(\cdot)$ verifying the relations} $(\ref{II4})$.

$i)$ {\it The control $\displaystyle (u_{\alpha})_{\alpha =\overline{1,m}}$
transfers the phase $(t_0,x_0)$ into the phase $(t,y)$ if and only if}
$\quad
\displaystyle
\chi(t_0,t)y-x_0
=
\int\limits_{\gamma_{t_0,t}}^{}
\chi(t_0,s)N_{\alpha}(s)
u_{\alpha}(s)
\mbox{d}s^{\alpha}.
$

$ii)$ {\it The phase $(t_0,x_0)$ transfers into the phase $(t,y)$ if and only if}
$$
x_0-\chi(t_0,t)y\in \mathcal{V}(t_0,t).
$$

$iii)$ {\it The phase $(t_0,x_0)$ is controllable if and only if
$\exists t\in D$, with $t^{\alpha}> t^{\alpha}_0$,
$\forall \alpha$ such that}
$ x_0\in \mathcal{V}(t_0,t)$.

$iv)$ {\it Let $t_0, t \in D$, with $t_0^{\alpha}<t^{\alpha}$, $\forall \alpha$.
The PDE system is completely controllable
from the multitime $t_0$ into the multitime $t$ if and only if}
\,\,$\mathcal{V}(t_0,t)=\mathbb{R}^n$.

\end{theorem}

\begin{proof}
The first statement $i)$ is a consequence of the formula (\ref{solCauchy}) and
the properties of the fundamental matrix. The second statement $ii)$ follows
from $i)$ and the definition of $\mathcal{V}(t_0,t)$. From $ii)$ (taking $y=0$) and from the Definition \ref{ddefinitia2},
one obtains $iii)$ and $iv)$.
\end{proof}

\begin{proposition}
\label{ppropozitia1}
\rm{(see \cite{8}, \cite{3})} {\it Let us suppose that the matrices
$(M_{\alpha}(\cdot))_{\alpha = \overline{1,m}}$ verify the relations}
$(\ref{II4})$, $\forall t\in D, \,\, \forall \alpha,\beta=\overline{1,m}$.
{\it We fix $t_0 \in D$. For each
$v \in \mathbb{R}^n$ and $\alpha=\overline{1,m}$,
we introduce the functions}
$$
u_{\alpha,v}: D \to \mathbb{R}^k,
\quad
u_{\alpha,v}(s)
=
N_{\alpha}^\top(s)\chi(t_0,s)^\top v,
\quad
\forall
s \in D.
$$

{\it The following statements are equivalent}

$i)$ {\it For any $v \in \mathbb{R}^n$,
$(u_{\alpha,v})_{\alpha=\overline{1,m}}$ is a control for the PDE system} $(\ref{II3})$.

$ii)$ {\it For any $\alpha, \beta =\overline{1,m}$,
the following relations are satisfied on the set $D$:}
\begin{equation}
\label{II6}
\begin{split}
M_{\alpha}N_{\beta}
N_{\beta}^\top
+
\frac{\partial N_{\alpha}}{\partial s^{\beta}}
N_{\alpha}^\top
+
N_{\alpha}
\frac{\partial N_{\alpha}^\top}{\partial s^{\beta}}
+
N_{\beta}N_{\beta}^\top
M^\top_{\alpha}
\\
=
M_{\beta}N_{\alpha}
N_{\alpha}^\top
+
\frac{\partial N_{\beta}}{\partial s^{\alpha}}
N_{\beta}^\top
+
N_{\beta}
\frac{\partial N_{\beta}^\top}{\partial s^{\alpha}}
+
N_{\alpha}N_{\alpha}^\top
M^\top_{\beta}.
\end{split}
\end{equation}

$iii)$ {\it The curvilinear integral
$\quad
\displaystyle
\int\limits_{\gamma}^{}
\chi(t_0,s)N_{\alpha}(s)
N_{\alpha}^\top(s)\chi(t_0,s)^\top
\mbox{d}s^{\alpha}
$
\, is path independent on $D$}.
\end{proposition}

\subsection{Quadratic affine forms on Hilbert spaces}

In this Section we shall reformulate two well-known Theorems regarding
the quadratic affine forms on a Hilbert space. These Theorems will be used in
the next Sections to obtain new results concerning the controllability
of multitime PDE systems.

\begin{theorem}
\label{tteorema1}
{\it Let $\mathcal{H}$ be a real Hilbert space and $T: \mathcal{H} \to \mathcal{H}$
be a linear, continuous, self-adjoint, positive semidefinite operator. For each $w \in \mathcal{H}$,
we consider the quadratic affine form}
\begin{equation*}
F_w: \mathcal{H} \to \mathbb{R},
\quad
F_w(v)=
\langle
T(v),v
\rangle
-
2
\langle
w,v
\rangle,
\quad
\forall
v \in \mathcal{H}.
\end{equation*}

$i)$ {\it If $v_0 \in \mathcal{H}$ is a local minimum point of $F_w$, then $T(v_0)=w$}.

$ii)$ {\it If $v_0 \in \mathcal{H}$ satisfies $T(v_0)=w$, then $v_0$ is a global minimum point of $F_w$}.

$iii)$ {\it If there exists a local maximum point $v_0 \in \mathcal{H}$
for $F_w$, then $w=0$ and $T(v)=0$, $\forall v \in \mathcal{H}$,
i.e., the function $F_w$ is identically zero (and, evidently, in this case, any point is a global minimum point)}.

$iv)$ {\it $v_0$ is a local extremum point of $F_w$ if and only if $v_0$ is a global minimum point for $F_w$}.
\end{theorem}

\begin{theorem}
\label{tteorema3} {\it One consider $\mathcal{H}$
as a real Hilbert space of finite dimension $n$. Let
$T: \mathcal{H} \to \mathcal{H}$ be an self-adjoint, positive semidefinite, linear operator.
Let $A$ be the matrix of $T$, associated in an arbitrary basis.

For each $w \in \mathcal{H}$,
we consider the function}
\begin{equation*}
F_w: \mathcal{H} \to \mathbb{R},
\quad
F_w(v)=
\langle
T(v),v
\rangle
-
2
\langle
w,v
\rangle,
\quad
\forall
v \in \mathcal{H}.
\end{equation*}
{\it Then the following statements are equivalent:

$i)$ $T$ is positive definite.

$ii)$ $T$ is bijective (equivalent to $\mbox{\rm rank}\, A=n$).

$iii)$ For any $w \in \mathcal{H}$, the function $F_w$ has at most a minimum point.

$iv)$ For each $w \in \mathcal{H}$, the function $F_w$ has at least a minimum point.

$v)$ For each $w \in \mathcal{H}$, there exists a unique minimum point for $F_w$}.

\end{theorem}

\begin{remark}
\label{oobservatia1}
\rm According to the Theorem $\ref{tteorema1}$,
it follows that in the Theorem $\ref{tteorema3}$,
the minimum point can be understood either a local minimum point
or a global minimum point, or a local extremum point.
\end{remark}


\section{Increasing curves and curvilinear integrals}

Here we analyse the notion of increasing curve (respectively, decreasing curve)
and we prove some propositions concerning the curvilinear integrals.

\begin{lemma}
\label{llema1} {\it Let $f : [a,b] \to \mathbb{R}$ be a
derivable function with the property that for any $\xi_1$, $\xi_2 \in [a,b]$, with $\xi_1<\xi_2$,
we have $f(\xi_1)<f(\xi_2$) (i.e., $f$ is strictly increasing).
Suppose there exists $c\in [a,b]$, such that $f'(c)=0$. Then there exists a sequence $\eta_p \in [a,b]$, with $f'(\eta_p) \neq 0$, $\forall p$,
such that $\displaystyle \lim_{p \to \infty } \eta_{p} = c$.}
\end{lemma}

\begin{definition}
\label{ddefinitia5}
\rm
Let $t_0,t \in \mathbb{R}^m$, such that $t_0^{\alpha} \leq t^{\alpha}$
(respectively $t_0^{\alpha} \geq t^{\alpha}$),
$\forall \alpha=\overline{1,m}$ and let
$\gamma : [a,b] \to \mathbb{R}^m$, $\gamma(\tau)=
\Big( \gamma^{\alpha}(\tau) \Big)_{\alpha=\overline{1,m}}$, with
$\gamma(a)=t_0$, $\gamma(b)=t$, a piecewise $\mathcal{C}^1$ curve. We say that
$\gamma$ {\it increases} (respectively {\it decreases}) from $t_0$ to $t$, if for any $\tau_1,\tau_2\in [a,b]$,
with $\tau_1<\tau_2$, we have
$$
\left\{
  \begin{array}{ll}
    \gamma^{\alpha}(\tau_1)<\gamma^{\alpha}(\tau_2)\,
    (\hbox{respectively }
    \gamma^{\alpha}(\tau_1)> \gamma^{\alpha}(\tau_2)),
         & \hbox{if } \quad t_0^{\alpha} \neq t^{\alpha}; \\
    \gamma^{\alpha}(\tau_1)=\gamma^{\alpha}(\tau_2),
         & \hbox{if } \quad t_0^{\alpha}=t^{\alpha}.
  \end{array}
\right.
$$
\end{definition}

\begin{remark}
\label{oobservatia3}
\rm Let $t_0,t \in D$, such that $t_0^{\alpha} \leq t^{\alpha}$
(respectively $t_0^{\alpha} \geq t^{\alpha}$),
$\forall \alpha=\overline{1,m}$. There exists at least one $\mathcal{C}^1$ curve,
included in $D$, which increases (respectively decreases)
from the multitime $t_0$ to the multitime $t$. For example, the straight line
segment $[t_0,t]$, parameterized by:\,
$
\gamma : [0,1] \to D$,
$
\gamma(\tau )=( 1-\tau ) t_0+\tau t$,
$
\forall \tau \in [0,1],
$\,
is increasing (respectively decreasing) and included in $D$ (since $D$ is a convex set).
\end{remark}

\begin{lemma}
\label{llema2} {\it Let $t_0,t \in D$, such that $t_0^{\alpha} \leq t^{\alpha}$
(respectively $t_0^{\alpha} \geq t^{\alpha}$),
$\forall \alpha=\overline{1,m}$ and let
$\gamma : [a,b] \to D$, with
$\gamma(a)=t_0$, $\gamma(b)=t$, a piecewise $\mathcal{C}^1$ curve which increases
(respectively decreases) from $t_0$ to $t$}.

{\it If $P_1,P_2, \dots, P_m : D \to [0,\infty)$ are continuous functions and}
$\displaystyle
\int\limits_{\gamma}
P_{\alpha}(s)
\, \mbox{d}s^{\alpha}
=0,
$
{\it then, for any $\alpha=\overline{1,m}$, with
$t^{\alpha} \neq t_0^{\alpha}$, and for any
$\tau \in [a,b]$, we have $P_{\alpha}(\gamma(\tau))=0$}.
\end{lemma}

\begin{proof}
There exist the real numbers $\tau_0, \tau_1, \dots, \tau_q$,
such that
$$
a=\tau_0 < \tau_1< \dots < \tau_q=b,
\quad
\mbox{with }\,
q \geq 1, \,\,
q \in \mathbb{N},
$$
and the curve $\gamma$ is of class $\mathcal{C}^1$ on each subinterval
$[\tau_j,\tau_{j+1}]$, $j=\overline{0,q-1}$.

We know that\,
$\displaystyle
\sum_{j=0}^{q-1}
\int\limits_{\tau_j}^{\tau_{j+1}}
\sum_{\alpha=1}^m
P_{\alpha}(\gamma(\tau))
\dot{\gamma^{\alpha}}(\tau)
\, \mbox{d}\tau
=0.
$
But, for $\alpha$, with $t^{\alpha}=t_0^{\alpha}$, the function
$\gamma^{\alpha}(\cdot)$ is constant on $[\tau_0,\tau_{q}]$
(see the Definition \ref{ddefinitia5}). Consequently,
$\dot{\gamma^{\alpha}}(\tau)=0$,
$\forall \tau \in  [\tau_0,\tau_{q}]$, hence
$$
\sum_{j=0}^{q-1}
\int\limits_{\tau_j}^{\tau_{j+1}}
\sum_{\alpha \mbox{ with } t^{\alpha}\neq t^{\alpha}_0}
P_{\alpha}(\gamma(\tau))
\dot{\gamma^{\alpha}}(\tau)
\, \mbox{d}\tau=0.
$$
Since $\gamma^{\alpha}$ is an increasing curve, on each subinterval
$[\tau_j,\tau_{j+1}]$, we have $\dot{\gamma^{\alpha}}(\tau) \geq 0$. We deduce that
each from the $q$ integrals, which are terms in the foregoing sum, are $\geq 0$.
In fact, each integral vanishes as term in a sum equal to zero, i.e.,
$$
\int\limits_{\tau_j}^{\tau_{j+1}}
\sum_{\alpha \mbox{ with } t^{\alpha}\neq t^{\alpha}_0}
P_{\alpha}(\gamma(\tau))
\dot{\gamma^{\alpha}}(\tau)
\, \mbox{d}\tau=0,
\quad
\forall j=\overline{0,q-1}.
$$
Since the functions that appear in the sum under the integral
are continuous and greater or equal to zero,
it follows that for each index $j=\overline{0,q-1}$, and for any $\alpha$,
with $t^{\alpha}\neq t^{\alpha}_0$, we have
$
P_{\alpha}(\gamma(\tau))
\dot{\gamma^{\alpha}}(\tau)
=0$,
$
\forall
\tau \in [\tau_j,\tau_{j+1}].
$
It follows that for each $\alpha$, with
$t^{\alpha}\neq t^{\alpha}_0$, we have
\begin{equation*}
P_{\alpha}(\gamma(\tau))
=0,
\quad
\forall
\tau \in [\tau_j,\tau_{j+1}],
\mbox{ with }
\dot{\gamma^{\alpha}}(\tau)
\neq 0.
\tag{$\ast \ast$}
\end{equation*}

Let $\alpha$ with $t^{\alpha} \neq t^{\alpha}_0$,
$\alpha$ arbitrarily, but fixed.
Let $c \in [\tau_j,\tau_{j+1}]$, with
$\dot{\gamma^{\alpha}}(c)=0$.

We apply Lemma \ref{llema1} to the function
$\gamma^{\alpha}(\cdot)$, on the interval
$[\tau_j,\tau_{j+1}]$. Hence there exists a sequence $\eta_p \in [\tau_j,\tau_{j+1}]$, with
$\dot{\gamma^{\alpha}}(\eta_p) \neq 0, \forall p$,
such that $\displaystyle \lim_{p \to \infty} \eta_p =c$.

Since $\dot{\gamma^{\alpha}}(\eta_p) \neq 0$,
according to $(\ast \ast)$, we have \,
$ P_{\alpha}(\gamma(\eta_p))=0$, \, $\forall p$.
Consequently
$$
0
=
\lim_{p \to \infty}
P_{\alpha}(\gamma(\eta_p))
=
P_{\alpha}(\gamma(c)).
$$
In this way we showed that $P_{\alpha}(\gamma(c))=0$, \,
$\forall c \in [\tau_j,\tau_{j+1}]$, with
$\dot{\gamma^{\alpha}}(c)=0$. Hence, via $(\ast \ast)$, it follows
$P_{\alpha}(\gamma(\tau))=0$, \,
$\forall \tau \in [\tau_j,\tau_{j+1}]$; and
since $j=\overline{0,q-1}$ is arbitrarily, we find\,
$
P_{\alpha}(\gamma(\tau))
=0$,
$
\forall
\tau \in [a,b].
$
\end{proof}

\section{Im - gramian space and conditions for \\multitime controllability}

In the papers \cite{8}, \cite{3} it was defined the controllability gramian, $\mathcal{C}(t_0,t)$, but only
in the case that the relations $(\ref{II6})$ hold. In this Section we shall extend
the definition to the general situation. But in this case, when
the relations $(\ref{II6})$ are not indispensable true, the gramian will depend on the curve
$\gamma$. It appears the notion of $\gamma$ - gramian matrix.
Instead of controllability gramian image we shall introduce the Im - gramian space $\mathcal{W}(t_0,t)$. 
We formulate necessary conditions of controllability expressed
using the space $\mathcal{W}(t_0,t)$.

\begin{definition}
\label{ddefinitia4}
\rm
Let us suppose that the matrices
$M_\alpha(\cdot)$ verify the relations $(\ref{II4})$.

$i)$ Let $\gamma : [a,b] \to D$ be a piecewise $\mathcal{C}^1$ curve,
with fixed origin $t_0=\gamma (a)$. The matrix
$$
\mathcal{C}_{\gamma}:=
\int\limits_{\gamma}^{}
\chi(t_0,s)N_{\alpha}(s)
N_{\alpha}^\top(s)\chi(t_0,s)^\top
\mbox{d}s^{\alpha}
$$
is called $\gamma$ - {\it gramian matrix} associated to the PDE system $(\ref{II3})$.

$ii)$ Suppose that, for any
$\alpha,\beta=\overline{1,m}$, the relations $(\ref{II6})$ are
true. According to the Proposition $\ref{ppropozitia1}$, in this case,
the curvilinear integral from $i)$ depends only on the  ends points and
not on the curve joining these ends. Let $t_0,t\in D$
and $\gamma_{t_0,t} : [a,b] \to D$ be a piecewise $\mathcal{C}^1$ curve, with
$\gamma_{t_0,t}(a)=t_0$ and $\gamma_{t_0,t}(b)=t$. The matrix
$$
\mathcal{C}(t_0,t):=
\int\limits_{\gamma_{t_0,t}}^{}
\chi(t_0,s)N_{\alpha}(s)
N_{\alpha}^\top(s)\chi(t_0,s)^\top
\mbox{d}s^{\alpha}
$$
is called the {\it controllability gramian}.
\end{definition}

\begin{remark}
\label{oobservatia2}
\rm Let us suppose that the matrices
$M_\alpha(\cdot)$ verify the relations $(\ref{II4})$.
Let $\gamma : [a,b] \to D$ be a piecewise $\mathcal{C}^1$ curve, and
$t_0:=\gamma (a)$, $t:=\gamma (b)$.
Let $\gamma^{-} : [a,b] \to D$, $\gamma^{-}(\tau)
=\gamma (a+b-\tau)$, $\forall \tau \in [a,b]$. Obviously,
we have $\gamma^{-}(a)=t$, $\gamma^{-}(b)=t_0$.

One verifies immediately
$$
\chi(t,t_0)
\mathcal{C}_{\gamma}
\chi(t,t_0)^\top
=
-
\mathcal{C}_{\gamma^{-}},
\quad
\mbox{hence}
\quad
\mbox{rank}\, (\mathcal{C}_{\gamma})
=
\mbox{rank}\, (\mathcal{C}_{\gamma^{-}}).
$$
Taking into account that $\chi(t,t_0)^\top$ is invertible, it follows that the equality
$\mbox{Im}(\mathcal{C}_{\gamma^{-}})$
$=\chi(t,t_0)\mbox{Im}(\mathcal{C}_{\gamma})$ holds.
\end{remark}

\begin{remark}
\label{oobservatia20}
\rm
Let $t_0,t \in \mathbb{R}^m$, such that $t_0^{\alpha} \leq t^{\alpha}$
(respectively $t_0^{\alpha} \geq t^{\alpha}$),
$\forall \alpha=\overline{1,m}$ and let
$\gamma : [a,b] \to \mathbb{R}^m$, with
$\gamma(a)=t_0$, $\gamma(b)=t$, a piecewise $\mathcal{C}^1$ curve.

One remarks that $\gamma$ increases (respectively decreases) from $t_0$ to $t$,
if and only if $\gamma^{-}$ decreases (respectively increases) from $t_0$ to $t$.
\end{remark}

\begin{definition}
\label{ddefinitia55}
\rm Suppose that the matrices
$M_\alpha(\cdot)$ verify the relations
$(\ref{II4})$. Let $t_0,t \in D$, such that $t_0^{\alpha} \leq t^{\alpha}$
(respectively $t_0^{\alpha} \geq t^{\alpha}$),
$\forall \alpha=\overline{1,m}$. We consider the set
\begin{equation}
\label{Wt0t}
\mathcal{W}(t_0,t):=
\bigcap_{\gamma_{t_0,t}}
\mbox{Im}(\mathcal{C}_{\gamma_{t_0,t}})
\end{equation}
where the intersection is taken over all curves $\gamma_{t_0,t}$, of piecewise $\mathcal{C}^1$ class, included in $D$, which increases (respectively decreases)
from $t_0$ to $t$. The set $\mathcal{W}(t_0,t)$ is a vector subspace
of $\mathbb{R}^n$ and we shall call it the {\it Im - gramian space}.
\end{definition}

\begin{remark}
\label{oobservatia200}
\rm
From the Remarks $\ref{oobservatia2}$ and $\ref{oobservatia20}$, it follows that, in the conditions of Definition $\ref{ddefinitia55}$, we have
\begin{equation}
\label{Wt0tWtt0}
\mathcal{W}(t,t_0)
=\chi(t,t_0)\mathcal{W}(t_0,t).
\end{equation}
\end{remark}

\begin{remark}
\label{oobservatia30}
\rm
Furthermore, if for any $\alpha,\beta=\overline{1,m}$, the relations
$(\ref{II6})$ are true, then $\mathcal{W}(t_0,t)=\mbox{Im}(\mathcal{C}(t_0,t))$\,
(see also the Definition $\ref{ddefinitia4}$).
\end{remark}

\begin{proposition}
\label{ppropozitia2} {\it Suppose that the matrices $M_\alpha(\cdot)$
verify the relations} $(\ref{II4})$.
{\it Let $t_0,t \in D$, such that $t_0^{\alpha} \leq t^{\alpha}$
(respectively, $t_0^{\alpha} \geq t^{\alpha}$),
$\forall\alpha=\overline{1,m}$}. {\it Then}
\begin{equation}
\label{IIincluziune}
\mathcal{V}(t_0,t)
\subseteq
\mathcal{W}(t_0,t).
\end{equation}
\end{proposition}

\begin{proof}
Let $\gamma$ be a piecewise $\mathcal{C}^1$  curve,
included in $D$, which increases (respectively, decreases)
from $t_0$ to $t$. It is sufficient to prove the inclusion
\begin{equation}
\label{II7}
\mathcal{V}(t_0,t)
\subseteq
\mbox{Im}(\mathcal{C}_{\gamma}),
\end{equation}
which is equivalent to
$\displaystyle
(\mathcal{V}(t_0,t))^\bot
\supseteq
(\mbox{Im}(\mathcal{C}_{\gamma}))^\bot
=
\mbox{Ker} ((\mathcal{C}_{\gamma})^\top)$.

We have $v \in \mbox{Ker} ((\mathcal{C}_{\gamma})^\top)
\Longleftrightarrow
((\mathcal{C}_{\gamma})^\top)v=0
\Longleftrightarrow
v^\top \mathcal{C}_{\gamma}=0.$ Hence
$v^\top \mathcal{C}_{\gamma}v=0$, equivalent to
$$
\int\limits_{\gamma}
v^\top
\chi(t_0,s)
N_{\alpha}(s)
N_{\alpha}^\top(s)
\chi(t_0,s)^\top
v
\, \mbox{d}s^{\alpha}
=
\int\limits_{\gamma}
\Big \|
v^\top
\chi(t_0,s)
N_{\alpha}(s)
\Big \|^2
\, \mbox{d}s^{\alpha}
=0.
$$
We apply Lemma \ref{llema2}, for $\displaystyle P_{\alpha}(s)=\Big \|
v^\top
\chi(t_0,s)
N_{\alpha}(s)
\Big \|^2$ and it follows that for any
$\alpha$, with $t^{\alpha} \neq t_0^{\alpha}$, and
for any $\tau$, we have\,\,
$\displaystyle
v^\top
\chi(t_0,\gamma(\tau))
N_{\alpha}(\gamma(\tau))
=0$.

Let now $(u_{\alpha}(\cdot))_{\alpha=\overline{1,m}}$ be an arbitrary control. We get
$$
\left \langle
\int\limits_{\gamma}^{}
\chi(t_0,s)N_{\alpha}(s)
u_{\alpha}(s)
\, \mbox{d}s^{\alpha} \,
; \,
v
\right \rangle
=
v^\top
\int\limits_{\gamma}^{}
\chi(t_0,s)N_{\alpha}(s)
u_{\alpha}(s)
\, \mbox{d}s^{\alpha}
$$
$$
=
\int\limits_{\gamma}
\sum_{\alpha \mbox{ with } t^{\alpha} \neq t^{\alpha}_0}
v^\top
\chi(t_0,s)N_{\alpha}(s)
u_{\alpha}(s)
\, \mbox{d}s^{\alpha}
=0.
$$
Hence $v \in (\mathcal{V}(t_0,t))^\bot$,
whence we obtain the inclusion
$$
\mbox{Ker} ((\mathcal{C}_{\gamma})^\top)
\subseteq
(\mathcal{V}(t_0,t))^\bot.
$$
\end{proof}

From the Theorem \ref{tteorema5} and Proposition \ref{ppropozitia2} one obtains
immediately the next Theorem (in which we give necessary conditions
for controllability):

\begin{theorem}
\label{tteorema6} {\it Let us consider the PDE system $(\ref{II3})$, with the
matrix functions $M_\alpha(\cdot)$ verifying the relations} $(\ref{II4})$.

{\it $i)$ Let $t_0,t \in D$, such that
$t_0^{\alpha} \leq t^{\alpha}$ (respectively, $t_0^{\alpha} \geq t^{\alpha}$),
$\forall\alpha=\overline{1,m}$. If the phase $(t_0,x_0)$
transfers into the phase $(t,y)$, then}
$$
x_0-\chi(t_0,t)y \in \mathcal{W }( t_0,t ).
$$

$ii)$ {\it If the phase $(t_0,x_0)$ is controllable, then
$\exists t\in D$, with $t^{\alpha}> t^{\alpha}_0$,
$\forall \alpha$, such that}\, $x_0\in \mathcal{W }( t_0,t )$.

$iii)$ {\it Let $t_0, t \in D$, with $t_0^{\alpha}<t^{\alpha}$, $\forall \alpha$.
If the PDE system $(\ref{II3})$ is completely controllable from the
multitime $t_0$ into the multitime $t$, then} $\mathcal{W }(t_0,t )=\mathbb{R}^n$.
\end{theorem}

\begin{remark}
\label{oobservatiaTeza}
\rm
If the hypotheses of Proposition \ref{ppropozitia2} are completed by the relations
$(\ref{II6})$, then the inclusion $(\ref{II7})$ becomes equality (\cite{3}), i.e.,
$\mathcal{V}(t_0,t)=\mbox{Im}(\mathcal{C}(t_0,t))=\mathcal{W}(t_0,t)$. In
these conditions, one deduces a theorem,
similar to the Theorem \ref{tteorema6}, but
in which one obtains necessary and sufficient conditions for controllability.
\end{remark}

Generally, the inclusion (\ref{IIincluziune}) is strictly. We shall justify
this statement by the following example:
\begin{example}
\label{IInuincluziune}
\rm
We consider: $m=2$, $n=2$, $k=2$,
\begin{equation*}
M_1=M_2=0,
\quad \quad
N_2=0,
\quad \quad
N_1(s^1,s^2)=
\left(
  \begin{array}{cc}
    s^2 & 0  \\
    0 & s^2  \\
  \end{array}
\right)
=s^2I_2,
\end{equation*}
$$
a)\, D=\mathbb{R}^2;
\quad \quad \quad
b)\, D=
\left \{
\left(
t^1,t^2
\right)\in \mathbb{R}^2 \,
\big |\,
t^2>0
\right \}
=
\mathbb{R} \times (0,\infty).
$$
We shall show that for any $t_0,t\in D$, with $t_0^1<t^1$, $t_0^2<t^2$, in the case $a)$,
we have $\mathcal{V}(t_0,t)=0$ and $\mathcal{W}(t_0,t)=\mathbb{R}^2=\mathcal{M}_{2,1}(\mathbb{R})$, while
in the case $b)$, the equalities $\mathcal{V}(t_0,t)=\mathcal{W}(t_0,t)=\mathbb{R}^2$ hold.
\end{example}

We have $\chi(s_0,s)=I_2$,
$\forall (s_0,s) \in \mathbb{R}^2
\times \mathbb{R}^2$.

$a)$ According to the Definition \ref{ddefinitia1}, the pair $(u_1, u_2)$
is a control if and only if (\ref{II5}) is true, i.e.,
\begin{equation*}
\frac{\partial }{\partial s^{2}}
\left(
N_{1}(s^1,s^2)
\right)
u_{1}(s^1,s^2)
+
N_{1}(s^1,s^2)
\frac{\partial }{\partial s^{2}}
\left(  u_{1}(s^1,s^2)  \right)
=0,
\quad
\forall (s^1,s^2) \in \mathbb{R}^2,
\end{equation*}
or
\begin{equation}
\label{II4.8}
\frac{\partial }{\partial s^{2}}
\left(
s^2u_{1}(s^1,s^2)
\right)
=
0,
\quad
\forall (s^1,s^2)\in \mathbb{R}^2,
\end{equation}
equivalent to the fact that there exists a $\mathcal{C}^1$ function $f:\mathbb{R} \to \mathbb{R}^2=
\mathcal{M}_{2,1}(\mathbb{R})$, such that
\begin{equation*}
s^2u_{1}(s^1,s^2)=f(s^1),
\quad
\forall (s^1,s^2)\in \mathbb{R}^2.
\end{equation*}
Setting $s^2=0$, one obtains $f(s^1)=0$,
$\forall s^1\in \mathbb{R}$. Hence
\begin{equation*}
s^2u_{1}(s^1,s^2)=0,
\quad
\forall (s^1,s^2)\in \mathbb{R}^2,
\end{equation*}
whence it follows $u_{1}(s^1,s^2)=0$,
$\forall (s^1,s^2)\in \mathbb{R}^2$,
with $s^2\neq 0$. From the continuity of $u_1$,
we deduce that
\begin{equation*}
u_{1}(s^1,s^2)=0,
\quad
\forall (s^1,s^2)\in \mathbb{R}^2.
\end{equation*}
For $u_1=0$ and for any $u_2$,
the relation (\ref{II4.8}) is verified.

Hence the set of controls is made of the pairs $(u_1,u_2)$, with $u_1=0$
and $u_2:\mathbb{R}^2 \to \mathbb{R}^2=
\mathcal{M}_{2,1}(\mathbb{R})$ as an arbitrary $\mathcal{C}^1$ function.

Let $\gamma$ be a piecewise $\mathcal{C}^1$ curve which increases from
the multitime $t_0$ to the multitime $t$. Since $u_1$ and $N_2$ vanish, it follows
\begin{equation*}
\int\limits_{\gamma }^{}
\chi(t_0,s)N_{1}(s)
u_{1}(s)
\mbox{d}s^{1}
+
\chi(t_0,s)N_{2}(s)
u_{2}(s)
\mbox{d}s^{2}
=0,
\end{equation*}
hence $\mathcal{V}(t_0,t)=0$ (here it does no matter that $\gamma$ is increasing,
since the curvilinear integral is path independent).

According to the definition,\,
$\displaystyle
\mathcal{C}_{\gamma}
=
\int\limits_{\gamma }^{}
(s^2)^2I_2
\mbox{d}s^{1}
=
\Big(
\int\limits_{\gamma }^{}
(s^2)^2
\mbox{d}s^{1}
\Big)
I_2$.
We show that $\displaystyle \int\limits_{\gamma } (s^2)^2
\mbox{d}s^{1} \neq 0$. Indeed, suppose that we would have $\displaystyle \int\limits_{\gamma } (s^2)^2
\mbox{d}s^{1} = 0$. According to Lemma \ref{llema2}, it follows that for any
$\tau$, we have $(\gamma^2(\tau))^2=0$, hence $\gamma^2$
is constant (zero), whence we deduce that $t_0^2=t^2$, which is false.

We deduce that the matrix $\mathcal{C}_{\gamma}$ has the rank $2$. Hence
$\mbox{Im} ( \mathcal{C}_{\gamma} )=\mathbb{R}^2$; it follows the equality $\mathcal{W}(t_0,t)=\mathbb{R}^2$.

$b)$ In the same manner as in the case $a)$, one deduces that $(u_1, u_2)$
is a control if and only if the relation (\ref{II4.8})
is true on the set $D$. Let
$v \in \mathbb{R}^2=\mathcal{M}_{2,1}(\mathbb{R})$. We select
$$
u_1(s^1,s^2)=\frac{1}{s^2(t^1-t_0^1)} \cdot v,
\quad
u_2(s^1,s^2)=0,
\quad
\forall (s^1,s^2)
\in \mathbb{R} \times (0,\infty)
=D.
$$
We remark immediately that for any pair $(u_1, u_2)$,
the relations (\ref{II4.8}) is true, hence $(u_1, u_2)$ is a control.

Let $\gamma$ be a piecewise $\mathcal{C}^1$ curve,
included in $D$, which joins $t_0$ to $t$. Let us determine
$\mathcal{V}(t_0,t)$. One can easily obtain
\begin{equation*}
\int\limits_{\gamma }^{}
\chi(t_0,s)N_{1}(s)
u_{1}(s)
\mbox{d}s^{1}
+
\chi(t_0,s)N_{2}(s)
u_{2}(s)
\mbox{d}s^{2}
=v.
\end{equation*}
Since $v$ is arbitrary, it follows
$\mathcal{V}(t_0,t)=\mathbb{R}^2$, and from (\ref{IIincluziune}) it follows
also the equality $\mathcal{W}(t_0,t)=\mathbb{R}^2$.
\begin{remark}
\label{oobservatia300}
\rm
For the system of Example $\ref{IInuincluziune}$, case $a)$, $D=\mathbb{R}^2$, the inclusion $(\ref{IIincluziune})$ is strictly.
Also, to the same extent, one deduces that the converse of the Theorem $\ref{tteorema6}$ is not always true.
Indeed, for any $t_0,t \in \mathbb{R}^2$, with $t_0^1<t^1$, $t_0^2<t^2$, we have $\mathcal{V}(t_0,t)=0$,
hence from the Theorem $\ref{tteorema5}$, $iii)$, it follows that no state $(t_0,x_0)$,
with $x_0 \neq 0$, is controllable; however $\mathcal{W}(t_0,t)=\mathbb{R}^2=\mathbb{R}^n$.

For the same system, considered in the case $b)$,
$D=\mathbb{R} \times (0, \infty)$, for any $t_0,t \in D$, with
$t_0^1<t^1$, $t_0^2<t^2$, we have $\mathcal{V}(t_0,t)=\mathbb{R}^2=\mathbb{R}^n$,
whence according to the Theorem $\ref{tteorema5}$, $iv)$,
it follows that the system is completely controllable.

One remarks that for the controllability, it is important the domain on which
we consider the matrix functions $M_\alpha(\cdot)$ and $N_\alpha(\cdot)$ which define the system.
\end{remark}

\section{A decision functional for \\multitime controllability}

Roughly speaking there are essentially two types of methods to study
the controllability of linear PDE, namely direct methods and dual methods.

Suppose that the matrices $M_{\alpha}(\cdot)$ verify the relations $(\ref{II4})$. We
consider the adjoint (dual) PDE system of the
PDE system $(\ref{II3})$, i.e.,
\begin{equation}
\label{II8}
\frac{\partial \varphi}{\partial t^{\alpha}}
=-
(M_{\alpha}(t))^{\top}\varphi,
\quad
\forall \alpha=\overline{1,m}
\end{equation}
Let $t_0 \in D$ and $v\in\mathbb{R}^n$.
The Cauchy problem $\{ (\ref{II8}), \,  \varphi(t_0)=v \}$ has the solution
\begin{equation}
\label{II9}
\varphi_v(\, \cdot \, , t_0) : D \to \mathbb{R}^n,
\quad
\varphi_v(s , t_0)
=
\chi(t_0,s)^{\top}v,
\quad
\forall s \in D.
\end{equation}
Denote by $\mathcal{S}$, the set of solutions of the PDE system $(\ref{II8})$, i.e.,
\begin{equation*}
\mathcal{S}=
\left \{
\varphi : D \to \mathbb{R}^n
\Bigm |
\varphi
\mbox{ solution of the system (\ref{II8})}
\right \}
=
\left \{
\varphi_v(\, \cdot \, , t_0)
\Bigm |
v \in
\mathbb{R}^n
\right \}.
\end{equation*}
First, we remark that $\mathcal{S}$ is a real vector space.

Let $x_0 \in \mathbb{R}^n$, $t_0,t\in D$,
and $\gamma_{t_0,t}$ be a piecewise $\mathcal{C}^1$ curve, included
in $D$, which joins the multitime $t_0$ to the multitime $t$
(covered from $t_0$ to $t$). We define the functional\,
$F_{\gamma_{t_0,t}} (\, \cdot \, , x_0;t_0,t) : \mathcal{S} \to \mathbb{R}$,
\begin{equation}
\label{II10}
F_{\gamma_{t_0,t}}
(\varphi, x_0;t_0,t)
=
\int\limits_{\gamma_{t_0,t}}
\Big \|
N_{\alpha}(s)^{\top}\varphi(s)
\Big \|^2
\, \mbox{d}s^{\alpha}
-2
\langle
x_0, \varphi (t_0)
\rangle,
\quad
\forall
\varphi \in \mathcal{S},
\end{equation}
and we call it the {\it controllability $\gamma_{t_0,t}$ - functional}.

The function $L_{t_0} : \mathbb{R}^n \to \mathcal{S}$,
$L_{t_0}(v)=\displaystyle \varphi_v(\,\cdot \, , t_0)$ is obvious an isomorphism of
vector spaces. Its inverse is
$L_{t_0}^{-1} : \mathcal{S} \to \mathbb{R}^n $,
$L_{t_0}^{-1}(\varphi)=\varphi(t_0)$,
$\forall \varphi \in \mathcal{S}$.

Now define the function
\begin{equation}
\label{II11}
\widetilde{F}_{\gamma_{t_0,t}}
(\, \cdot \, , x_0;t_0,t)
: \mathbb{R}^n \to \mathbb{R},
\,\,\,\,
\widetilde{F}_{\gamma_{t_0,t}}
(v , x_0;t_0,t)
=
F_{\gamma_{t_0,t}}
(L_{t_0}(v), x_0;t_0,t).
\end{equation}
Let us write $\widetilde{F}_{\gamma_{t_0,t}}
(\, \cdot \, , x_0;t_0,t)$ by a formula showing that it is
in fact a quadratic affine form. In this sense, we have

\begin{equation}
\label{II12}
\widetilde{F}_{\gamma_{t_0,t}}
(v , x_0;t_0,t)
=
\left \langle
\mathcal{C}_{\gamma_{t_0,t}}v \, ,
v
\right \rangle
-2
\left \langle
x_0, v
\right \rangle,
\quad
\forall
v \in \mathbb{R}^n.
\end{equation}

\begin{remark}
\label{oobservatia4}
\rm
From $(\ref{II11})$ and from the fact that $L_{t_0}$ is a bijective function,
we deduce that $v$ is an extremum point (respectively minimum, maximum) for
$\widetilde{F}_{\gamma_{t_0,t}}
(\, \cdot \, , x_0;t_0,t)$ if and only if $L_{t_0}(v)=
\displaystyle \varphi_v(\,\cdot \, , t_0)$ is
an extremum point (respectively minimum, maximum) for
$F_{\gamma_{t_0,t}} (\, \cdot \, , x_0;t_0,t)$.
And conversely: $\varphi$ is an extremum point (respectively minimum, maximum) for
$F_{\gamma_{t_0,t}} (\, \cdot \, , x_0;t_0,t)$
if and only if
$L_{t_0}^{-1}(\varphi)=\varphi(t_0)$ is
an extremum point (respectively minimum, maximum) for
$\widetilde{F}_{\gamma_{t_0,t}}
(\, \cdot \, , x_0;t_0,t)$. Hence $L_{t_0}$ is a bijection between the extremum point set (respectively minimum, maximum) of
$\widetilde{F}_{\gamma_{t_0,t}}
(\, \cdot \, , x_0;t_0,t)$
and the extremum point set (respectively minimum, maximum) of
$F_{\gamma_{t_0,t}} (\, \cdot \, , x_0;t_0,t)$.
The notions of extremum (respectively minimum, maximum) are understood globally, for the time being.

In order to speak of local extremum points, we shall give a topology on $\mathcal{S}$.
Since $L_{t_0}$ is an isomorphism of vector spaces, it follows that $\mathcal{S}$
has finite dimension, namely $n$. We endow $\mathcal{S}$ with the topology induced by an arbitrary norm. There exists a norm on $\mathcal{S}$, for example,
$\| \varphi \|_{\mathcal{S}}=
 \| L_{t_0}^{-1} (\varphi) \|=
\| \varphi (t_0)\|$.
Since $\mathcal{S}$ has finite dimension, any two norms are equivalent, hence they induce the same topology.
Now, $L_{t_0}$ is a homeomorphism between $\mathbb{R}^n$ and
$\mathcal{S}$. We deduce that $L_{t_0}$ is a bijection between the set of local extremum points
(respectively local minimum, local maxim) of
$\widetilde{F}_{\gamma_{t_0,t}}
(\, \cdot \, , x_0;t_0,t)$
and the set of local extremum points
(respectively local minimum, local maxim) of
$F_{\gamma_{t_0,t}} (\, \cdot \, , x_0;t_0,t)$.
\end{remark}

\vspace{0.2 cm}
Throughout, in the sequel, we suppose that the matrices $M_{\alpha}(\cdot)$ verify
the relations $(\ref{II4})$, on the set $D$.

Let $t_0, t \in D$, with
$t_0^{\alpha} \leq t^{\alpha}$, $\forall \alpha$
and let $\gamma_{t_0,t} : [a,b] \to D$ be a piecewise $\mathcal{C}^1$ curve, which increase from $t_0$ to $t$. Then
\begin{equation}
\label{II13}
\left \langle
\mathcal{C}_{\gamma_{t_0,t}}v \, ,
v
\right \rangle
=
\int\limits_{\gamma_{t_0,t}}
\Big \|
N_{\alpha}(s)^{\top}\chi(t_0,s)^{\top}v
\Big \|^2
\, \mbox{d}s^{\alpha}
\geq 0,
\quad
\forall
v \in \mathbb{R}^n,
\end{equation}
The inequality $(\ref{II13})$ holds
since $\dot{\gamma}_{t_0,t}^{\alpha} (\tau) \geq 0$,
$\forall \tau$ (eventually, for a finite number of points, we have lateral derivative).

The matrix $\mathcal{C}_{\gamma_{t_0,t}}$ is symmetric. From now on and from $(\ref{II13})$,
we deduce that the function
$\widetilde{F}_{\gamma_{t_0,t}} (\, \cdot \, , x_0;t_0,t)$ takes the form of the functional
$F_w$, of Theorems \ref{tteorema1} and \ref{tteorema3}.
Here, $\mathcal{H}=\mathbb{R}^n$, $\dim \mathcal{H} =n < \infty$,
$T(v)=\mathcal{C}_{\gamma_{t_0,t}}v$; $T$ is linear, hence also continuous,
since $\mathcal{H}$ has
finite dimension; $T$ is self-adjoint and positive semidefinite (of $(\ref{II13})$). Also $w=x_0$.

Consequently, we can apply the Theorems \ref{tteorema1} and \ref{tteorema3},
for $\widetilde{F}_{\gamma_{t_0,t}} (\, \cdot \, , x_0;t_0,t)$.

It follows that the local extremum points of
$\widetilde{F}_{\gamma_{t_0,t}} (\, \cdot \, , x_0;t_0,t)$,
if they exist, are in fact global minimum points. According to the Remark \ref{oobservatia4},
it follows that the same thing happen also for $F_{\gamma_{t_0,t}} (\, \cdot \, , x_0;t_0,t)$.
Therefore, in the sequel we shall refer only to (global) minimum points.

From the Theorem \ref{tteorema1} and the Remark \ref{oobservatia4},
we find immediately

\begin{proposition}
\label{ppropozitia3} {\it Let $t_0, t \in D$, with
$t_0^{\alpha} \leq t^{\alpha}$, $\forall \alpha$,
let $\gamma_{t_0,t} : [a,b] \to D$ be a piecewise $\mathcal{C}^1$ curve,
which increases from $t_0$ to $t$, and let $x_0\in \mathbb{R}^n$}.
{\it Then the following statements are equivalent}:

$i)$ {\it $F_{\gamma_{t_0,t}} (\, \cdot \, , x_0;t_0,t)$
has at least one minimum point}.

$ii)$ {\it $\widetilde{F}_{\gamma_{t_0,t}} (\, \cdot \, , x_0;t_0,t)$
has at least a minimum point}.

$iii)$ $x_0 \in \mbox{\rm Im}(\mathcal{C}_{\gamma_{t_0,t}})$.
\end{proposition}

From the Theorem \ref{tteorema3} and the Remark \ref{oobservatia4},
it follows immediately

\begin{proposition}
\label{ppropozitia4} {\it Let $t_0, t \in D$, with
$t_0^{\alpha} \leq t^{\alpha}$, $\forall \alpha$
and let $\gamma_{t_0,t} : [a,b] \to D$ be a piecewise $\mathcal{C}^1$ curve,
increasing from $t_0$ to $t$}. {\it Then
the following statements are equivalent}:

$i)$ {\it For any $x_0 \in \mathbb{R}^n$,
there exists a unique minimum point for the functional}
$F_{\gamma_{t_0,t}} (\, \cdot \, , x_0;t_0,t)$.

$ii)$ {\it For any $x_0 \in \mathbb{R}^n$,
there exists a unique minimum point for the functional}
$\widetilde{F}_{\gamma_{t_0,t}} (\, \cdot \, , x_0;t_0,t)$.

$iii)$ $\mbox{\rm rank}\, (\mathcal{C}_{\gamma_{t_0,t}})=n$.
\end{proposition}

From the Propositions \ref{ppropozitia3}, \ref{ppropozitia4}
and Theorem \ref{tteorema6}, it follows directly:

\begin{theorem}
\label{tteorema7} {\it Let us consider
the PDE system $(\ref{II3})$, with the matrix functions
$M_\alpha(\cdot)$ verifying the relations} $(\ref{II4})$.

{\it $i)$ Let $t_0,t \in D$, such that
$t_0^{\alpha} \leq t^{\alpha}$,
$\forall\alpha=\overline{1,m}$. If the phase $(t_0,x_0)$
transfers into the phase $(t,y)$, then} {\it for any piecewise $\mathcal{C}^1$ curve ${\gamma}_{t_0,t}$, included in $D$,
increasing from $t_0$ to $t$, it follows that
$
F_{\gamma_{t_0,t}} (\, \cdot \, , x_0-\chi(t_0,t)y ;t_0,t)
$
and
$
\widetilde{F}_{\gamma_{t_0,t}} (\, \cdot \, , x_0-\chi(t_0,t)y;t_0,t)
$
have global minimum points}.

$ii)$ {\it If the phase $(t_0,x_0)$ is controllable, then
$\exists t\in D$, with $t^{\alpha}> t^{\alpha}_0$,
$\forall \alpha$, such that for any piecewise $\mathcal{C}^1$ curve ${\gamma}_{t_0,t}$,
included in $D$, increasing from $t_0$ to $t$, it follows that
$
F_{\gamma_{t_0,t}} (\, \cdot \, , x_0;t_0,t)
$
and
$
\widetilde{F}_{\gamma_{t_0,t}} (\, \cdot \, , x_0;t_0,t)
$
have global minimum points}.

$iii)$ {\it Let $t_0, t \in D$, with $t_0^{\alpha}<t^{\alpha}$, $\forall \alpha$.
If the PDE system $(\ref{II3})$ is completely controllable
from the multitime $t_0$ into the multitime $t$, then
for any piecewise $\mathcal{C}^1$ curve $\gamma_{t_0,t}$, included
in $D$, increasing from $t_0$ to $t$, and
for any $x_0 \in \mathbb{R}^n$, there exists a unique (global) minimum point for}
$
F_{\gamma_{t_0,t}} (\, \cdot \, , x_0;t_0,t).
$

{\it Analogously for} $\widetilde{F}_{\gamma_{t_0,t}} (\, \cdot \, , x_0;t_0,t)$.
\end{theorem}

\section{Controllability functional}
We consider the PDE system $(\ref{II3})$,
for which the matrix functions
$M_{\alpha}(\cdot)$ and $N_{\alpha}(\cdot)$ verify
the relations $(\ref{II4})$ and $(\ref{II6})$, on the set $D$.

Let $x_0\in \mathbb{R}^n$ and $t_0,t \in D$.
Let $\gamma_{t_0,t}$ be a piecewise $\mathcal{C}^1$ curve,
included in $D$. According to $(\ref{II12})$, we have
\begin{equation*}
\widetilde{F}_{\gamma_{t_0,t}}
(v , x_0;t_0,t)
=
\left \langle
\mathcal{C}_{\gamma_{t_0,t}}v \, ,
v
\right \rangle
-2
\left \langle
x_0, v
\right \rangle,
\quad
\forall
v \in \mathbb{R}^n.
\end{equation*}
But we have seen (Proposition \ref{ppropozitia1},
Definition \ref{ddefinitia4})
that in this case, the functional $\mathcal{C}_{\gamma_{t_0,t}}$
does not depend on the curve $\gamma_{t_0,t}$, but only on the ends $t_0$, $t$. It will be denoted by
$\mathcal{C}(t_0,t)$. From here and from $(\ref{II12})$, it follows that also $\widetilde{F}_{\gamma_{t_0,t}} (\, \cdot \, , x_0;t_0,t)$
does not depend on the curve $\gamma_{t_0,t}$, but only on the ends $t_0$, $t$. We denote
\begin{equation}
\label{II14}
\widetilde{F}
(v , x_0;t_0,t)
:=
\widetilde{F}_{\gamma_{t_0,t}} (\, \cdot \, , x_0;t_0,t)
=
\left \langle
\mathcal{C}(t_0,t)v \, ,
v
\right \rangle
-2
\left \langle
x_0, v
\right \rangle,
\,\,\,
\forall
v \in \mathbb{R}^n.
\end{equation}
Since $(\ref{II11})$
\begin{equation*}
F_{\gamma_{t_0,t}} (\varphi , x_0;t_0,t)
=
\widetilde{F}_{\gamma_{t_0,t}} ( L_{t_0}^{-1}( \varphi ), x_0;t_0,t)
=
\widetilde{F} ( L_{t_0}^{-1}( \varphi ), x_0;t_0,t),
\quad
\forall
\varphi \in \mathcal{S},
\end{equation*}
it follows that
$F_{\gamma_{t_0,t}} (\, \cdot \, , x_0;t_0,t)$ depends only on the ends $t_0$, $t$, and not on the curve $\gamma_{t_0,t}$. We denote
\begin{equation*}
F (\, \cdot \, , x_0;t_0,t)
:=
F_{\gamma_{t_0,t}} (\, \cdot \, , x_0;t_0,t),
\end{equation*}
and we call it the {\it controllability functional}.

Let $t_0,t \in D$, with $t_0^{\alpha}<t^{\alpha}$, $\forall \alpha$.
From the Theorem \ref{tteorema5} and Remark \ref{oobservatiaTeza}, it follows
that if the relations (\ref{II6})
hold, then the PDE system $(\ref{II3})$ is completely controllable from $t_0$ to $t$
if and only if
$\mbox{rank}\, \mathcal{C}(t_0,t)=n$;
from the Theorem \ref{tteorema1}, Theorem \ref{tteorema3},
Remark \ref{oobservatia4},
we find immediately:

\begin{theorem}
\label{tteorema8} {\it Let us consider
the PDE system $(\ref{II3})$, with the matrix functions
$M_\alpha(\cdot)$, $N_{\alpha}(\cdot)$ verifying the
relations $(\ref{II4})$ and $(\ref{II6})$, on $D$}.

{\it $i)$ Let $t_0,t \in D$, $t_0\neq t$,
$t_0^{\alpha}\leq t^{\alpha}$,
$\forall \alpha$. Then the phase $(t_0,x_0)$
transfers to the phase $(t,y)$ if and only if\,
$F (\, \cdot \, , x_0-\chi(t_0,t)y ;t_0,t)$
has at least a minimum point}.

$ii)$ {\it The phase $(t_0,x_0)$ is controllable if and only if
$\exists t\in D$, with $t^{\alpha}> t^{\alpha}_0$, $\forall \alpha$, such that\,
$F (\, \cdot \, , x_0 ;t_0,t)$
has at least a minimum point}.

$iii)$ {\it Let $t_0,t \in D$, with $t_0^{\alpha}<t^{\alpha}$, $\forall \alpha$.
The PDE system $(\ref{II3})$ is completely controllable
from $t_0$ to $t$ if and only if for any $x_0 \in \mathbb{R}^n$, there exists a
unique global minimum point, for the functional $F (\, \cdot \, , x_0 ;t_0,t)$.}

$iv)$ {\it Let $t_0,t \in D$, with $t_0^{\alpha}<t^{\alpha}$, $\forall \alpha$.
The PDE system $(\ref{II3})$ is completely controllable
from $t_0$ to $t$ if and only if
for any $x_0 \in \mathbb{R}^n$, the functional $F (\, \cdot \, , x_0 ;t_0,t)$
has at least a minimum point.}

$v)$ {\it Let $t_0,t \in D$, with $t_0^{\alpha}<t^{\alpha}$, $\forall \alpha$.
The PDE system $(\ref{II3})$ is completely controllable
from $t_0$ to $t$ if and only if
for any $x_0 \in \mathbb{R}^n$, the functional $F(\, \cdot \, , x_0 ;t_0,t)$
has at most a minimum point.}

{\it Analogously for} $\widetilde{F} (\, \cdot \, , x_0;t_0,t)$.

\end{theorem}

\section{An unbounded extention of \\controllability functional}

Let us give an example of a complete controllable PDE system, with the matrices
$M_{\alpha}$, $N_{\alpha}$, verifying the relations
$(\ref{II4})$ and $(\ref{II6})$.
The Theorem \ref{tteorema7} says that for $t_0$, $t$, with $t_0^{\alpha}< t^{\alpha}$,
and for any curve $\gamma_{t_0,t}$, which increases from $t_0$ to $t$,
the controllability $\gamma_{t_0,t}$ - functional, has a unique global minimum point
on the set $\mathcal{S}$.

The Formula (\ref{II10}), which defines the functional $F_{\gamma_{t_0,t}}(\, \cdot \, , x_0;t_0,t)$, has sense also on some spaces $\mathcal{S}_1$, which
contains $\mathcal{S}$, as would be the space $\mathcal{C}^r(D;\mathbb{R}^n)$ with $r\in \{ 0,1,2 \}$. This means that we can introduce a new functional
\begin{equation*}
J_{\gamma_{t_0,t}}
(\, \cdot \, , x_0;t_0,t)
: \mathcal{S}_1 \to \mathbb{R},
\end{equation*}
\begin{equation}
\label{II18}
J_{\gamma_{t_0,t}}
(\varphi, x_0;t_0,t)
=
\int\limits_{\gamma_{t_0,t}}
\Big \|
N_{\alpha}(s)^{\top}\varphi(s)
\Big \|^2
\, \mbox{d}s^{\alpha}
-2
\langle
x_0, \varphi (t_0)
\rangle,
\quad
\forall
\varphi \in \mathcal{S}_1.
\end{equation}
The functional $J_{\gamma_{t_0,t}}(\, \cdot \, , x_0;t_0,t)$ is obviously
an extension of
$F_{\gamma_{t_0,t}}(\, \cdot \, , x_0;t_0,t)$.

We select $\mathcal{S}_1=\mathcal{C}^2(D;\mathbb{R}^n)$. In the case of PDE system in our example,
we shall show that for any $x_0 \neq 0$, for any $t_0$, $t$,
with $t_0^{\alpha}< t^{\alpha}$, $\forall \alpha$, and for any
curve $\gamma_{t_0,t}$, which increases from $t_0$ to $t$, the
extention $J_{\gamma_{t_0,t}}(\, \cdot \, , x_0;t_0,t)$
is unbounded from below.
Hence the Theorems \ref{tteorema7} and \ref{tteorema8}
are no longer valid
on spaces including strictly the set $\mathcal{S}$.

Now, suppose
$
m=2,\,\,n=2,\,\,k=1,\,\, D=\mathbb{R}^2,
$
$$
M_1(s)=M_2(s)=\left(
          \begin{array}{cc}
            0 & 0 \\
            0 & 0 \\
          \end{array}
        \right),
\,\,
N_1(s)=\left(
      \begin{array}{c}
        1 \\
        0 \\
      \end{array}
    \right),
\,\,
N_2(s)=\left(
      \begin{array}{c}
        0 \\
        1 \\
      \end{array}
    \right),
\,\,
\forall
s\in \mathbb{R}^2.
$$

The relations $(\ref{II4})$ and $(\ref{II6})$ hold.
We have $\chi(s_1,s_2)=I_2$, $\forall s_1,s_2 \in \mathbb{R}^2$.

Let $t_0=(t_0^1,t_0^2)$, $t=(t^1,t^2)$, with
$t_0^1<t^1$, $t_0^2<t^2$ and $x_0=(a,b)^{\top}$.

For $v=(v_1,v_2)^{\top} \in \mathbb{R}^2
=\mathcal{M}_{2,1}(\mathbb{R})$, we find
$
\varphi_v (s,t_0)=\chi(t_0,s)^{\top}v=v$,
$
\forall  s \in \mathbb{R}^2;
$
$$
N_1^{\top}(s)\varphi_v (s,t_0)=v_1,
\quad
N_2^{\top}(s)\varphi_v (s,t_0)=v_2,
\quad
\forall  s \in \mathbb{R}^2.
$$
\begin{equation*}
F ( \, \cdot \, , x_0;t_0,t )
: \mathcal{S} \to \mathbb{R},
\quad
F \Big( \varphi_v (\, \cdot \, ,t_0), x_0;t_0,t \Big)
=
\end{equation*}
\begin{equation*}
=
(t^1-t_0^1)
\cdot
\left(
v_1 - \frac{a}{t^1-t_0^1}
\right )^2
+
(t^2-t_0^2)
\cdot
\left(
v_2 - \frac{b}{t^2-t_0^2}
\right )^2
\end{equation*}
\begin{equation*}
-
\frac{a^2}{t^1-t_0^1}
-
\frac{b^2}{t^2-t_0^2}
\geq
-
\frac{a^2}{t^1-t_0^1}
-
\frac{b^2}{t^2-t_0^2},
\end{equation*}
with equality if and only if $\displaystyle v_1 =\frac{a}{t^1-t_0^1}$ \,
and $\displaystyle v_2 =\frac{b}{t^2-t_0^2}$.

Hence, for any $t_0=(t_0^1,t_0^2)$, $t=(t^1,t^2)$, with
$t_0^1<t^1$, $t_0^2<t^2$ and any $x_0=(a,b)^{\top}$, the functional
$F ( \, \cdot \, , x_0;t_0,t )$ has a unique global minimum point,
which is $\varphi_{v_0} (\, \cdot \, ,t_0)$,
where $v_0=\displaystyle \left(
\frac{a}{t^1-t_0^1}\, , \frac{b}{t^2-t_0^2} \right)$.

From the Theorem \ref{tteorema8}, it follows that the system is completeley controllable.

We denote
$
\mathcal{S}_1=
\left \{
\varphi : \mathbb{R}^2=D \to \mathbb{R}^2
\Bigm |
\varphi
\mbox{ of class }
\mathcal{C}^{2}
\right \}.
$
Obviously $\mathcal{S} \subseteq \mathcal{S}_1$.

Let $\gamma_{t_0,t}$ be an increasing curve from $t_0$ to $t$. We select
$x_0=(a,b)^{\top} \neq (0,0)^{\top}$. Define $J_{\gamma_{t_0,t}} ( \, \cdot \, , x_0;t_0,t )$, by the
formula $(\ref{II18})$. We obtain an extension of the controllability functional
to the space $\mathcal{S}_1$.

Let $q>0$. If $a+b \neq 0$, we select $c=a+b$. If $b=-a$, we select $c=a$;
we have $c \neq 0$ since $x_0 \neq 0$. We choose
$$
\varphi_1(s^1,s^2)=
c \cdot
\sqrt{\frac{q}{1+q^2(s^1+s^2-t_0^1-t_0^2)^2}},
\quad
\forall (s^1,s^2) \in \mathbb{R}^2.
$$
If $a+b \neq 0$, we consider the function $\varphi_2(\cdot)=\varphi_1(\cdot)$.
If $b=-a$, we take the function $\varphi_2(\cdot)=-\varphi_1(\cdot)$.
We select $\varphi(\cdot)=(\varphi_1(\cdot), \varphi_2(\cdot))^{\top}$.
Obviously, $\varphi \in \mathcal{S}_1$.

In the case $a+b \neq 0$, we obtain
\begin{equation*}
J_{\gamma_{t_0,t}} ( \varphi , x_0;t_0,t )
=(a+b)^2\,
\arctan
\Big(
q (t^1+t^2-t_0^1-t_0^2)
\Big)
-2
(a+b)^2 \sqrt{q}.
\end{equation*}

In case $b=-a \neq 0$, we find
\begin{equation*}
J_{\gamma_{t_0,t}} ( \varphi , x_0;t_0,t )
=a^2\,
\arctan
\Big(
q (t^1+t^2-t_0^1-t_0^2)
\Big)
-4
a^2 \sqrt{q}.
\end{equation*}
In both cases we have
$\displaystyle
\lim_{q \to \infty} J_{\gamma_{t_0,t}} ( \varphi , x_0;t_0,t )
=-\infty$.



\center

\end{document}